\documentclass[reqno]{amsart}
\usepackage{amscd}
\usepackage[dvips]{graphics}
\usepackage{color}
\usepackage{mathrsfs}
\usepackage{bm}

\def\beq{\begin{equation}}
\def\eeq{\end{equation}}
\def\ba{\begin{array}}
\def\ea{\end{array}}

\numberwithin{equation}{section}
\newenvironment{abs}{\textbf{Abstract}\mbox{  }}{ }
\newenvironment{key words}{\textbf{Keywords}\mbox{  }}{ }
\newtheorem{theorem}{Theorem}[section]
\newtheorem{definition}[theorem]{\textbf{Definition}}

\newtheorem{proposition}[theorem]{\textbf{Proposition}}
\newtheorem{lemma}[theorem]{Lemma}
\renewenvironment{proof}{\noindent{\textbf{Proof.}}}{\hfill$\Box$}
\theoremstyle{remark}
\newtheorem{remark}[theorem]{\textbf{Remark}}
\theoremstyle{plain}

\allowdisplaybreaks[4]

\begin{document}
\title[An integral type  Brezis-Nirenberg problem]{\textbf{An integral type  Brezis-Nirenberg problem on the Heisenberg group }}
\author  {Yazhou Han}

\address{Yazhou Han, Department of Mathematics, College of Science, China Jiliang University, Hangzhou, 310018, China}
\email{yazhou.han@gmail.com}


\date{}
\maketitle

\noindent
\begin{abs}
This paper is devoted to study a class of integral type Brezis-Nirenbreg problem on the Heisenberg group. It is a class of new nonlinear integral equations on the bounded domains of Heisenberg group and related to the CR Yamabe problems on the CR manifold. Based on the sharp Hardy-Littlewood-Sobolev inequalities, the nonexistence and existence results are obtained by Pohozaev type identity, variational method and blow-up analysis, respectively.
%
\end{abs}

\smallskip
\noindent\begin{key words}
Heisenberg group, Brezis-Nirenberg problem, Integral equations, Existence
\end{key words}

\smallskip
\noindent\textbf{Mathematics Subject Classification(2010).}
45G05, 35A01, 35B44 \indent
\section{Introduction\label{Section 1}}

CR manifold is a class of noncommutative geometry and arises from the study of the real hypersurface of complex manifold (see \cite{Dragomir-Tomassini2006, Folland-Stein1974} and the references therein). The complex structure of the real hypersurface, induced from the complex manifold, inspire many interesting geometric property and bring some new difficulties. Particularly, in the study of CR manifold, Heisenberg group $\mathbb{H}^n$ plays a similar role as $\mathbb{R}^n$ to Riemannian manifold. So, this paper is devoted to study the integral type Brezis-Nirenberg problem on the Heisenberg group.

Let us recall  the Sobolev inequality, Hardy-Littlewood-Sobolev (HLS) inequality and their corresponding equations on the Heisenberg group. Then, we will give our integral equations and our results. Involved notations can be seen in the Section \ref{Sec Preliminary}.

\subsection{Sobolev inequality on the Heisenberg group}
In 1980s, Jerison and Lee studied the CR Yamabe problem on CR manifolds in their series papers \cite{JL1983, JL1987, JL1988, JL1989}. As the idea of Yamabe, Trudinger and Aubin (see \cite{Y1960, T1968, A1976, LP1987}), 
the study of CR Yamabe problem is closely related to the sharp Sobolev inequality on the Heisenberg group, which can stated as:
\begin{equation}\label{Sobolev-Hn}
    S_{n,2}\left(\int_{\mathbb{H}^n} |f|^{2^*}d\xi \right)^\frac{2}{2^*} \le \int_{\mathbb{H}^n} |\nabla_H f|^2 d\xi,
\end{equation}
where $2^*=\frac{2Q}{Q-2}$ is Sobloev critical exponent and $S_{n,2}=\frac{4 n^2\pi}{{(2^{2n}n!)}^{1/(n+1)}}$ is the best constant. 
In \cite{JL1988}, Jerison and Lee used the Obata's idea and classified all extremal function, up to group translations, dilations and multiplication by a constant, as
\begin{equation}\label{Sobolev-ex1}
    U(\xi)=U(z,t)=((1+|z|^2)^2+t^2)^{-(Q-2)/4}.
\end{equation}
Recently, Frank and Lieb \cite{Frank-Lieb2012} gave a new proof to the extremal function by a rearrangement-free method.

Obviously, Sobolev inequality holds on any subset $\Omega\subset\mathbb{H}^n$. But, because of the classification of extremal functions on the above, we know that the best constant can not be attained if $\Omega\neq\mathbb{H}^n$. Namely, if $\Omega\neq\mathbb{H}^n$, there is not an energy minimizing solution to the Euler-Lagrange equation of \eqref{Sobolev-Hn}
\begin{equation}\label{Yamabe equ Hn}
    \begin{cases}
    -\Delta_H f=f^{\frac{Q+2}{Q-2}}\quad &\text{in}\quad \Omega,\\
   ~~~~~~~~~~~~ u=0\quad&\text{on}\quad \partial\Omega.
    \end{cases}
\end{equation}
Inspired by the above fact, on bounded domain $\Omega\subset\mathbb{H}^n$, the Brezis-Nireberg problems
\begin{equation}\label{Sob-critical}
  \begin{cases}
  -\Delta_H f=f^{2^*-1}+g(\xi,f), &\quad u>0, \quad \text{in}\quad \Omega,\\
  u=0&\quad\text{on}\quad \partial\Omega.
  \end{cases}
\end{equation}
were studied extensively, such as the results of \cite{ BR2017, Garofalo-Lanconelli1992, HCR2015, Loiudice2007, N1999, Wang2001}, etc.

%

\subsection{Hardy-Littlewood-Sobolev (HLS) inequlity on the Heisenberg group}
In \cite{Folland-Stein1974}, Folland and Stein studied the sigular integral operator and obtain the following HLS inequalities
\begin{equation}\label{class-HLS-1}
    \left|\int_{\mathbb{H}^n}\int_{\mathbb{H}^n} \overline{f(\xi)}g(\eta)|\eta^{-1}\xi|^{\alpha-Q}d\eta d\xi\right| \leq D(n,\alpha,p)\|f\|_{L^{q}(\mathbb{H}^n)}\|g\|_{L^p(\mathbb{H}^n)}
\end{equation}
where $f\in L^{q},\ g\in L^p$, $0<\alpha<Q$ and $\frac 1{q}+\frac 1p+\frac{Q-\alpha}Q=2$. In fact, the result is followed from the Proposition 8.7 of \cite{Folland-Stein1974} (see the Proposition \ref{pro 8.7 of FS} and its Remark in Section \ref{Sec Preliminary}).

Recently, for the diagonal case $p=q=\frac{2Q}{Q+\alpha}$, Frank and Lieb \cite{Frank-Lieb2012} identified the sharp constant $D(n,\alpha,p)$ and classified all extremal functions. We can summarize their results as
\begin{theorem}[Sharp HLS inequality on $\mathbb{H}^n$]\label{HLS Hn}
For $0<\alpha<Q$ and $p=\frac{2Q}{Q+\alpha}$. Then for any $f,g\in L^p(\mathbb{H}^n)$,
\begin{equation}\label{class-HLS}
    \left|\int_{\mathbb{H}^n} \int_{\mathbb{H}^n} \overline{f(\xi)}|\eta^{-1}\xi|^{-(Q-\alpha)} g(\eta) d\eta d\xi\right|\le D_{n,\alpha} ||f||_{L^{p}(\mathbb{H}^n)}||g||_{L^{p}(\mathbb{H}^n)}
\end{equation}
where
\begin{equation}\label{extreCon}
D_{n,\alpha}:=\left(\frac{\pi^{n+1}}{2^{n-1}n!}\right)^{(Q-\alpha)/Q} \frac{n!\Gamma(\alpha/2)}{\Gamma^2((Q+\alpha)/4)}.
\end{equation}
And equality holds if and only if
\begin{equation}\label{HLS-ex}
f(\xi)=c_1g(\xi)=c_2H(\delta_r(\zeta^{-1}\xi),
\end{equation}
for some $c_1, \ c_2\in\mathbb{C}$,\ $r>0$ and $\zeta\in\mathbb{H}^n$ (unless $f\equiv 0$ or $g\equiv 0$). Here $H$ is defined as
\begin{equation}\label{HLS-ex1}
    H(\xi)=H(z,t)=((1+|z|^2)^2+t^2)^{-(Q+\alpha)/4}.
\end{equation}
\end{theorem}

By a duality argument, based on the fundamental solution of sub-Laplace $-\Delta_H$ (see \cite{Folland1973} or \cite{Folland-Stein1974}), we see that the case $\alpha=2$ of Theorem \ref{HLS Hn} is equivalent to the sharp Sobolev inequality \eqref{Sobolev-Hn}. Hence, it is worth to study the   integral equation related to HLS inequality on the Heisenberg  group.

On the other hand, the integral form curvature problems was introduced and studied by Prof. Zhu in \cite{Zhu2016},  which gives an idea of global analysis to curvature problems. Dou and Zhu \cite{DZ2017} discussed the existence and nonexistence of positive solutions for an integral equation related to HLS inequality on the bounded domain of $\mathbb{R}^n$, and found some new phenomena which is different with partial differential equations.  This also implies that  integral equations have the independent research interests except using as tools for the study of differential equations.
Hence, in this work we will discuss the following  integral equation related to HLS inequality on the Heisenberg group.

For any smooth domain $\Omega \subset \mathbb{H}^n$ (for example, say, the  boundary is $C^2$), we consider
\begin{eqnarray*}
D_{n,\alpha}(\Omega)=\sup_{f\in L^\frac{2Q}{Q+\alpha}(\Omega)\backslash\{0\}} \frac{\int_{\Omega} \int_{\Omega} f(\xi)|\eta^{-1}\xi|^{-(Q-\alpha)} f(\eta) d\xi d\eta}{||f||^2_{L^\frac{2Q}{Q+\alpha}(\Omega)}}.
\end{eqnarray*}
Without loss of generality, we only need to consider non-negative functions.

Similar to CR Yamabe problem, we also can  investigate the fact
\begin{eqnarray}\label{energy}
 D_{n,\alpha} (\Omega)= D_{n,\alpha},
  \end{eqnarray}
  and $D_{n,\alpha} (\Omega) $ is not attained by any functions if $\Omega \ne \mathbb{H}^n$ (see Proposition \ref{prop2-1} below).

  Notice that the corresponding Euler-Lagrange equation for the maximizer (if the supremum is attained) is the following integral equation:
 \begin{equation}\label{HB-1}
f^\frac{Q-\alpha}{Q+\alpha}(\xi)=\int_\Omega \frac{f(\eta)}{|\eta^{-1}\xi|^{n-\alpha}}d\eta,\quad \xi\in \overline \Omega.
\end{equation}
We thus know that there is not an energy maximizing  solution to above integral equation.

Similar to Brezis-Nirenberg problem on $\mathbb{R}^n$ (see \cite{BN1983}), we will study the existence or non-existence of positive solutions to the above integral equation.
%
%
%
To this end, we consider the following general equation:

\begin{equation}\label{HB}
f^{q-1}(\xi)=\int_\Omega \frac{f(\eta)}{|\eta^{-1}\xi|^{Q-\alpha}}d\eta+\lambda \int_\Omega \frac{f(\eta)}{|\eta^{-1}\xi|^{Q-\alpha-1}}d\eta,\quad \xi\in \overline \Omega.
\end{equation}
 For simplicity, we denote $p_\alpha=\frac{2Q}{Q-\alpha},\  q_\alpha=\frac{2Q}{Q+\alpha}$ throughout  this paper.

Our main results is as follows.
\begin{theorem}\label{main}
Assume $\alpha \in (0, Q)$ and $\Omega\subset\mathbb{H}^n$ is a smooth bounded domain.\\
(1) For $\frac{2Q}{Q+\alpha}<q<2$ (subcritical case), there is a positive solution $f\in \Gamma^\alpha(\overline{\Omega})\subset C^{\alpha/2}(\overline\Omega)$ to equation \eqref{HB} for any given $\lambda \in \mathbb{R}$;\\
(2) For $q=\frac{2Q}{Q+\alpha}$ (critical case) and   $\lambda>0$, there is a positive solution $f\in \Gamma^\alpha(\overline{\Omega})$ to equation \eqref{HB};\\
(3) For $1<q\le\frac{2Q}{Q+\alpha}$ (critical and supercritical case) and   $\lambda\le 0$, if  $\Omega$ is a $\delta$-starshaped domain, then  there is only trivial non-negative $C^1$ continuous (up to the boundary) solution to \eqref{HB}.
\end{theorem}

\begin{remark}
In \cite{DZ2017}, Dou and Zhu discussed the integral equations \eqref{HB} on the bounded domain $\Omega\subset\mathbb{R}^n$ and proved the results similar to Theorem \ref{main} with a constraint $\alpha>1$. We give a different proof of compactness and regularity, it can extend $0<\alpha<1$.
\end{remark}

We organize the paper as follows: In section \ref{Sec Preliminary}, we introduce some notations and some known facts about Heisenberg group. In Section \ref{Sec subcri-critical non}, based on Frank and Lieb's result, namely Theorem \ref{HLS Hn}, we show the estimate \eqref{energy} and prove that $D_{n,\alpha}(\Omega)$ can not be attained. Then, by establishing a class of Pohozaev identity related to integral equations \eqref{HB}, we can prove the nonexistence result (part (3) of Theorem \ref{main}).  Section \ref{Sec supercritial exist} is devoted to the part (1) of Theorem \ref{main}. This is completed by two steps: existence result in $L^q(\Omega)$ (Lemma \ref{Bound- attained}) and regularity (Lemma \ref{regularity-lemma}). In section \ref{Sec critical exist}, we will give the existence for the critical exponent case (part (2) of Theorem \ref{main}) by the approximation method from subcritical to critical. To complete the proof, we need the uniform bound about the solutions of subcritical equations, which is obtained by the blow-up analysis (see Lemma \ref{lem3-3}).
%
%

\section{Preliminaries of $\mathbb{H}^n$}\label{Sec Preliminary}

In this section, we will state some notations and some known facts about the Heisenberg group $\mathbb{H}^n$. More details can be found in \cite{Folland1973, Folland-Stein1974, Folland1975} and the references therein.

The Heisenberg group $\mathbb{H}^n$ consists of the set
    $$\mathbb{C}^n \times \mathbb{R}=\{(z,t):z=(z_1, \cdots, z_n)\in \mathbb{C}^n,t\in  \mathbb{R}\}$$
with the multiplication law
    $$(z,t)(z',t')=(z+z', t+t'+2Im(z \cdot \overline{z'})),$$
where $z \cdot \overline{z'}=\sum_{j=1}^n z_j \overline{z_j'}.$ As usual, we write $z_j=x_j+\sqrt{-1}y_j$. In the sequel, we always denote the point of $\mathbb{H}^n$ by lowercase Greek characters such as $\xi=(z,t)=(x,y,t)$, $\eta=(w,s)=(u,v,s)$, etc.

The Lie algebra is spanned by
the left invariant vector fields
\begin{displaymath}
T=\frac{\partial}{\partial t}, X_j=\frac\partial{\partial x_j}+2y_j\frac\partial{\partial
t},\,Y_j=\frac\partial{\partial y_j}-2x_j\frac\partial{\partial
t},\,j=1,\cdots,n.
\end{displaymath}
The horizontal gradient and the sub-Laplacian are defined by $\nabla_{H}=(X_1,\cdots,X_n,Y_1,\cdots,Y_n)$ and
\begin{equation*}
    \triangle_H=\sum_{j=1}^n(X_j^2+Y_j^2),
\end{equation*}
respectively.

For any points $\xi=(z,t)$ and $\eta=(w,s)$, the norm function $|\xi|$ is defined as
    $$|\xi|=(|z|^4+t^2)^{\frac 14},$$
and, correspondingly, the distance between $\xi$ and $\eta$ is defined as $|\eta^{-1}\xi|$.
A family of dilations is defined as
    $$\delta_r(z,t)=(rz,r^2t),\quad\forall r>0,$$
and the homogeneous dimension with respect to the dilations is $Q=2n+2$.

Now, we state some basic facts on Heisenberg group as follows.

\begin{proposition}[(8.8) of \cite{Folland-Stein1974}]\label{pro compara norms}
For any $\xi\in\mathbb{H}^n$ with $|\xi|\leq 1$, then
    $$\|\xi\|\leq |\xi|\leq \|\xi\|^{1/2},$$
where $\|\cdot\|$ is the Euclidean norm.
\end{proposition}

\begin{proposition}[Lemma 8.9 of \cite{Folland-Stein1974}]\label{Pro triangle ineq}
There exists a constant $C\geq 1$ such that, for all $\xi,\eta\in\mathbb{H}^n$,
    $$|\xi+\eta|\leq C(|\xi|+|\eta|),\quad |\xi\eta|\leq C(|\xi|+|\eta|),$$
where $\xi+\eta$ represents the common vector adding.
\end{proposition}

We say that function $f$ is homogeneous of degree $\lambda$ if $f(\delta_r(z,t))=r^\lambda f(z,t)$, and that a distribution $F\in \mathscr{D}'$ is homogeneous of degree $\lambda$ if
    $$F(r^{-Q}g(\delta_{r^{-1}}(z,t)))=r^\lambda F(g).$$

\begin{proposition}[Proposition 8.1 of \cite{Folland-Stein1974}]\label{pro 8.1 of FS}
If $F\in\mathscr{D}'$ is homogeneous of degree $\lambda$, then $X_jF$ and $Y_j F$ are homogeneous of $\lambda-1$ for $1\leq j\leq n$.
\end{proposition}

Similar to Lemma 8.10 of \cite{Folland-Stein1974} and Proposition 1.15 of \cite{Folland1975}, we have the following result:
\begin{proposition}\label{pro 1.15 of Folland}
Let $f$ be a homogeneous function of degree $\lambda\ (\lambda\in\mathbb{R})$ which is $C^2$ away from $0$. There exists a constant $C>0$ such that
\begin{gather*}
    |f(\xi\eta)-f(\xi)|\leq C|\eta||\xi|^{\lambda-1},\quad\text{ whenever }\quad |\eta|\leq\frac 12|\xi|,\\
    |f(\xi\eta)+f(\xi\eta^{-1})-2f(x)|\leq C|\eta|^2|\xi|^{\lambda-2},\quad\text{ whenever }\quad |\eta|\leq\frac 12|\xi|.
\end{gather*}
\end{proposition}
\begin{proof}
The first result can be found in Lemma 8.10 of \cite{Folland-Stein1974}, which had be generalized to nilpotent Lie groups, i.e., Proposition 1.15 of \cite{Folland1975}. For completeness, we will give the proof of the second result.

By homogeneity, we can assume that $|\xi|=1$ and $|\eta|\leq 1/2$. Since $\xi-(\xi\eta-\xi)=\xi\eta^{-1}$, $f\in C^2(\mathbb{H}^n\backslash\{0\})$ and the smooth property of the mapping $\eta\mapsto\xi\eta$, we have
\begin{align*}
    &|f(\xi\eta)+f(\xi\eta^{-1})-2f(\xi)|\\
    =&|f(\xi+(\xi\eta-\xi))+f(\xi-(\xi\eta-\xi))-2f(\xi)|\\
    \leq & C\|\xi\eta-\xi\|\leq C\|\eta\|\leq C|\eta|,
\end{align*}
where the last inequality is deduced by Proposition \ref{pro compara norms}.
\end{proof}

Similar to \cite{Garofalo-Lanconelli1992}, we introduce the $\delta$-starshaped domain as follows:
\begin{definition}[Definition 2.1 of \cite{Garofalo-Lanconelli1992}]
Given a piecewise $C^1$ open set $\Omega\not\equiv\mathbb{H}^n$ and $(0,0)\in \Omega$, we say that it is $\delta$-starshaped with respect to origin if and only if
\begin{equation}\label{starshap condition 1}
    E\cdot \nu>0
\end{equation}
holds at every point of the boundary $\partial \Omega$, where $\nu$ is the outer normal to the boundary $\partial\Omega$ and the vector field $E$ is defined as
\begin{equation}\label{vector field}
    E=x\frac{\partial}{\partial x}+y\frac{\partial}{\partial y}+2t\frac{\partial}{\partial t}.
\end{equation}
\end{definition}

\begin{remark}\label{starshaped equiv}
It is easy to verify that condition \eqref{starshap condition 1} is equivalent to the following condition: if the point $\xi=(z,t)\in\Omega$, then
\begin{equation}\label{starshap condition 2}
    \delta_\lambda(\xi)=(\lambda z,\lambda^2 t)\in\Omega \quad\text{for}\quad \forall\lambda\in[0,1].
\end{equation}
\end{remark}

Following, we introduce the convolution on the $\mathbb{H}^n$ and their properties. More details can be found in \cite{Folland-Stein1974} and the references therein.

\begin{definition}[Convolution]
The convolution of two functions $f,g$ on $\mathbb{H}^n$ is defined by
    $$f*g(\xi)=\int f(\eta)g(\eta^{-1}\xi)d\eta=\int f(\xi\eta^{-1})g(\eta) d\eta.$$
If $f\in C_0^\infty$ and $G\in \mathscr{D}'$, we define the $C^\infty$ function $G*f$ and $f*G$ by
    $$G*f(\xi)=G(f(\eta^{-1}\xi)), \quad f*G(\xi)=G(f(\xi\eta^{-1}).$$
\end{definition}

\begin{definition}[Regular distribution]\label{def regular distribution}
A distribution $F$ is said to be regular if there exists a function $f$ which is $C^\infty$ on $\mathbb{H}^n\backslash\{0\}$ such that $F(g)=\int fg dV$ for all $g\in C_0^\infty(\mathbb{H}^n\backslash\{0\})$.
\end{definition}

\begin{proposition}[Proposition 8.7 of \cite{Folland-Stein1974}]\label{pro 8.7 of FS}
If $F$ is a regular homogeneous distribution of degree $\lambda,\ -Q<\lambda<0$, then the mapping $g\rightarrow g*F$ extends to a bounded mapping from $L^p$ to $L^q$, where $q^{-1}=p^{-1}-\lambda/Q-1$ provided $1<p<q<\infty$, and from $L^1$ to $L^{-Q/\lambda-\epsilon}(loc)$ for any $\epsilon>0$.
\end{proposition}

\begin{remark}
If the distribution $F$ is taken as $|\xi|^{\alpha-Q}(0<\alpha<Q)$, then for $1<p<q<+\infty$ and $\frac 1q=\frac 1p-\frac\alpha Q$, the above result can be specified as
    $$\|g*|\xi|^{\alpha-Q}\|_{L^q}\leq C\|g\|_{L^p}.$$
It  is just a  dual form  of HLS inequalities \eqref{class-HLS-1}.
\end{remark}

On $\mathbb{H}^n$, let $C^\beta,\ 0<\beta<\infty$ be the classical Lipschitz spaces of order $\beta$, namely, they are defined in terms of the Euclidean norm. In \cite{Folland-Stein1974}, they introduced the following family $\Gamma_\beta$ of Lipschitz spaces with respect to the norm $|\cdot|$ on the Heisenberg group.
\begin{definition}[Lipschitz spaces]
i) For $0<\beta<1$,
    $$\Gamma_\beta=\{f\in L^\infty\cup C:\ \sup_{\xi,\eta} \frac{|f(\xi\eta)-f(\xi)|} {|\eta|^\beta}<\infty\};$$
ii) For $\beta=1$,
    $$\Gamma_1=\{f\in L^\infty\cup C:\ \sup_{\xi,\eta} \frac{|f(\xi\eta)+f(\xi\eta^{-1})-2f(\xi)|} {|\eta|}<\infty\};$$
iii) For $\beta=k+\beta'$, where $k$ is a positive integer and $0<\beta'\le 1$,
    $$\Gamma_\beta=\{f\in L^\infty\cup C:\ f\in \Gamma_{\beta'}\ \text{and}\  Df\in \Gamma_{\beta'}\ \text{for all}\ D\in\mathscr{B}_k\}$$
where
$$\mathscr{B}_k=\{L_{a_1}L_{a_2}\cdots L_{a_j}:\ 1\le a_i\le 2n, i=1,2,\cdots,j,\ j\le k\}$$
with $L_j=X_j$ and $L_{j+n}=Y_j$ for $j=1,2,\cdots,n$.
\end{definition}
%

\begin{proposition}[Theorem 20.1 of \cite{Folland-Stein1974}]\label{thm 20.1 of FS}
$\Gamma_\beta\subset C^{\beta/2}(\text{loc})$ for $0<\beta<\infty$.
\end{proposition}

%
%
%
%

Notation: for any function $f(\xi)$ defined on $\Omega$, we always use $\tilde f(\xi)$ to represent its trivial extension in $\mathbb{H}^n$, namely,
\begin{eqnarray*}
    \tilde{f}(\xi)= \begin{cases}f(\xi)&\quad \xi\in\Omega,\\
    0&\quad \xi\in\mathbb{H}^n\backslash\Omega. \end{cases}
\end{eqnarray*}
And
    $$
    I_\alpha f(\xi)=\int_{\mathbb{H}^n}\frac {f(\eta)}{|\eta^{-1}\xi|^{Q-\alpha}} dy, \ \ \ \ I_{\alpha, \Omega} f(\xi)=\int_{\Omega}\frac {f(\eta)}{|\eta^{-1}\xi|^{Q-\alpha}} dy.
    $$
We also denote that $c,C$ different positive constant.

\section{Nonexistence for critical and supcritical case}\label{Sec subcri-critical non}

In this section, we mainly devote to discuss the nonexistence result for the critical case. Firstly, we derive energy estimate \eqref{energy}, and show that the  supremum $D_{n,\alpha}(\Omega) $ is not achieved by any function on any domain $\Omega \ne \mathbb{H}^n$. Then, we establish a Pohozaev type identity for integral equation, which deduce the third part of Theorem \ref{main}.

\begin{proposition}\label{prop2-1} For any domain $\Omega \subset \mathbb{H}^n$,
$D_{n,\alpha}(\Omega)=D_{n,\alpha}$; further the supremum $D_{n,\alpha}(\Omega)$  is not achieved by any function in $L^{q_\alpha}(\Omega)$ on any domain $\Omega \ne \mathbb{H}^n$.
\end{proposition}
\begin{proof}
If $f\in L^{q_\alpha}(\Omega)$, then $\tilde{f}\in L^{q_\alpha}(\mathbb{H}^n)$ . It follows that
\begin{align*}
    D_{n,\alpha}(\Omega) =&\sup\limits_{f\in L^{q_\alpha}(\Omega)\backslash\{0\}} \frac{\int_{\mathbb{H}^n}\int_{\mathbb{H}^n} \tilde{f}(\xi)|\eta^{-1}\xi|^{\alpha-Q} \tilde{f}(\eta) d\eta d\xi}{\|\tilde{f}\|_{L^{q_\alpha}(\mathbb{H}^n)}^2}\\
    \leq &\sup\limits_{g\in L^{q_\alpha}(\mathbb{H}^n)\backslash\{0\}} \frac{\int_{\mathbb{H}^n}\int_{\mathbb{H}^n} g(\xi)|\eta^{-1}\xi|^{\alpha-Q} g(\eta) d\eta d\xi}{\|g\|_{L^{q_\alpha}(\mathbb{H}^n)}^2}=D_{n,\alpha}.
\end{align*}

On the other hand,  recall  that $f(\xi)=H(\zeta^{-1}\xi)$ with $\zeta\in \mathbb{H}^n$ is an extremal function to the sharp HLS inequality in Theorem \ref{HLS Hn}, as well as its conformal equivalent class:
\begin{equation}\label{fe}
f_\epsilon(\xi)=\epsilon^{-\frac{Q+\alpha}2} H(\delta_{\epsilon^{-1}}(\zeta^{-1}\xi)),\quad \forall\epsilon>0.
\end{equation}
It is easy to verify
\[
\|I_\alpha f\|_{L^{p_\alpha}(\mathbb{H}^n)}=\|I_\alpha f_\epsilon\|_{L^{p_\alpha}(\mathbb{H}^n)},
\quad\|f\|_{L^{q_\alpha}(\mathbb{H}^n)} =\|f_\epsilon\|_{L^{q_\alpha}(\mathbb{H}^n)},
\]
and $f_\epsilon(\xi)$ satisfies integral equation
\begin{equation}\label{EL-equ-1}
f_\epsilon^\frac{Q-\alpha}{Q+\alpha}(\xi)=B\int_{\mathbb{H}^n} \frac{f_\epsilon(\eta)}{|\eta^{-1}\xi|^{Q-\alpha}}d\eta,
\end{equation}
where $B$ is a positive constant. Following, based on the extremal function $f_\epsilon$, we will choose a specific test function to prove that the reverse inequality holds too.

Choosse some point $\zeta\in \Omega$ and $R$ small enough so that $\Sigma_{R}(\zeta)\subset\Omega$, where $\Sigma_R(0)=\{\xi=(z,t)\in\mathbb{H}^n:\ |z|<R,\ |t|<R^2\}$ is a cylindrical set and $\Sigma_{R}(\zeta)=\zeta\circ\Sigma_R(0)$, and choose test function $g(\xi)\in L^{q_\alpha}(\mathbb{H}^n)$ as
 \begin{eqnarray*}
g(\xi)= \begin{cases}f_\epsilon(\xi) &\quad \xi\in \Sigma_{R}(\zeta)\subset\Omega,\\
 0&\quad \xi\in\mathbb{H}^n\backslash \Sigma_{R}(\zeta).
  \end{cases}
  \end{eqnarray*}
Then,
\begin{eqnarray*}
& &\int_{\Omega} \int_{\Omega} \frac{g(\xi)g(\eta)} {|\eta^{-1}\xi|^{Q-\alpha}} d\xi d\eta=\int_{\mathbb{H}^n} \int_{\mathbb{H}^n} \frac{ f_\epsilon(\xi) f_\epsilon(\eta)}{|\eta^{-1}\xi|^{Q-\alpha}} d\xi d\eta\\
& &-2\int_{\mathbb{H}^n} \int_{\Sigma_{R}^C(\zeta) } \frac{ f_\epsilon(\xi) f_\epsilon(\eta)}{|\eta^{-1}\xi|^{Q-\alpha}} d\xi d\eta
+\int_{\Sigma_{R}^C(\zeta) } \int_{\Sigma_{R}^C(\zeta) } \frac{ f_\epsilon(\xi) f_\epsilon(\eta)}{|\eta^{-1}\xi|^{Q-\alpha}} d\xi d\eta\\
&&=D_{n,\alpha}\|f_\epsilon\|^2_{L^{q_\alpha}(\mathbb{H}^n)}-I_1+I_2,
\end{eqnarray*}
where $\Sigma_{R}^C(\zeta)=\mathbb{H}^n\backslash \Sigma_{R}(\zeta)$ and
\[I_1:= 2\int_{\mathbb{H}^n} \int_{\Sigma_{R}^C(\zeta) } \frac{ f_\epsilon(\xi) f_\epsilon(\eta)}{|\eta^{-1}\xi|^{Q-\alpha}} d\xi d\eta, \quad
I_2:= \int_{\Sigma_{R}^C(\zeta) } \int_{\Sigma_{R}^C(\zeta) } \frac{ f_\epsilon(\xi) f_\epsilon(\eta)}{|\eta^{-1}\xi|^{Q-\alpha}} d\xi d\eta.
\]
It follows from \eqref{EL-equ-1} that
\begin{align*}
I_1 =
C\int_{\Sigma_{R}^C(\zeta)} f_\epsilon^{\frac{2Q}{Q+\alpha}}(\xi) d\xi =O(\frac{R}\epsilon)^{-Q},  \quad\quad\text{as}\quad\epsilon\to 0.
\end{align*}
By HLS inequality \eqref{class-HLS}, we have
\begin{eqnarray*}
I_2&\le&D_{n,\alpha}\|f_\epsilon\|^2_{L^{q_\alpha}(\Sigma_R^C(\zeta))}= O(\frac{R}\epsilon)^{-Q-\alpha}\quad\quad\text{as}\quad\epsilon\to 0.
\end{eqnarray*}
Combining the above, we arrive at
\begin{eqnarray*}
    D_{n,\alpha}(\Omega)&\ge &\frac{\int_{\Omega} \int_{\Omega} g(\xi)g(\eta)|\eta^{-1}\xi|^{\alpha-Q} d\xi d\eta}{\|g\|^2_{L^{q_\alpha}(\Omega)}}\\
    &=&\frac{D_{n,\alpha} \|f_\epsilon\|^2_{L^{q_\alpha}(\mathbb{H}^n)} -I_1+I_2} {\|f_\epsilon\|^2_{L^{q_\alpha}(\mathbb{H}^n)}}\\
    &=& D_{n,\alpha}-O(\frac{R}\epsilon)^{-Q}
\end{eqnarray*} with small enough $\epsilon$,
which yields $D_{n,\alpha}(\Omega)\ge D_{n,\alpha}$ as $\epsilon \to 0$.

Finally,  we show  that $D_{n,\alpha}(\Omega)$ is not achieved if $\Omega\neq\mathbb{H}^n$.
In fact, if $D_{n,\alpha}(\Omega)$ is attained by some function $u\in L^{q_\alpha}(\Omega)$, then $\tilde u\in L^{q_\alpha}(\mathbb{H}^n)$ would be an extremal function to the sharp HLS inequality on $\mathbb{H}^n$, which is impossible due to Theorem B.
\end{proof}
\medskip

Proposition \ref{prop2-1} indicates that there is no maximizing energy solution to \eqref{HB-1}. We shall show that there is not any non-trivial positive continuous solution to \eqref{HB-1} on any $\delta$-starshaped domain using the following Pohozaev identity. Without loss of generality, {\it in the rest of this section we always assume that the origin is in $\Omega$ and the domain is $\delta$-starshaped with respect to the origin.}

\begin{lemma}\label{poh}
If $f\in C^1(\overline \Omega)$ is a non-negative solution to
\begin{equation}\label{gen_equ}
f(\xi)=\int_\Omega\frac{f^{p-1}(\eta)}{|\eta^{-1}\xi|^{Q-\alpha}}d\eta+ \lambda \int_\Omega\frac{f^{p-1}(\eta)}{|\eta^{-1}\xi|^{Q-\alpha-1}}d\eta,\quad \xi\in \overline\Omega,
\end{equation}
where $p\ne 0$, $\lambda \in \mathbb{R}$, then
\begin{align}\label{Pohozaev identity}
    &(\frac{Q}{p}+\frac{\alpha-Q}2)\int_{\Omega}f^{p}(\xi)d\xi \nonumber\\ =& -\frac{\lambda}2 \int_{\Omega}\int_{\Omega} \frac{f^{p-1}(\xi)f^{p-1}(\eta)} {|\eta^{-1}\xi|^{Q-\alpha}}d\eta d\xi +\frac1{p}\int_{\partial \Omega}(E\cdot \nu) f^{p}(\xi)d\sigma,
\end{align}
where $\nu$ is the outward normal to $\partial \Omega$.
\end{lemma}
\begin{proof}
Denote by $\xi=(z,t)=(x,y,t),\ \eta=(w,s)=(u,v,s)$ and then
    $$\eta^{-1}\xi =(z-w,t-s-2\rm{Im}(w\cdot\overline{z}) =(z-w,t-s+2(yu-xv)).$$
Noting that $\Omega$ is a $\delta$-starshaped domain with respect to the origin and $Ef =\left.\frac{\partial}{\partial r} f(\delta_r(\xi))\right|_{r=1}$, we have that, by \eqref{gen_equ},
\begin{align}\label{formula 2.4}
    Ef(\xi)=\frac{\partial}{\partial r}\left(\int_\Omega \frac{f^{p-1}(\eta)}{|\eta^{-1}\delta_r(\xi)|^{Q-\alpha}}d\eta + \lambda \int_\Omega \frac{f^{p-1}(\eta)}{|\eta^{-1}\delta_r(\xi)|^{Q-\alpha-1}}d\eta\right)_{r=1}.
\end{align}
A direct calculation leads to
\begin{align}\label{formula 2.1}
\left.\frac{\partial}{\partial r} |\eta^{-1}\delta_r(\xi)|^4\right|_{r=1} =&4|z-w|^2[x(x-u)+y(y-v)]\nonumber\\ &+4(t-s+2(yu-xv))(t+(yu-xv)).
\end{align}
By \eqref{formula 2.4}, we have
\begin{align*}
    &\int_\Omega f^{p-1}(\xi) Ef(\xi) d\xi\\
    =&(\alpha-Q)\int_{\Omega}\int_{\Omega} \frac{f^{p-1}(\xi)f^{p-1}(\eta)} {|\eta^{-1}\xi|^{Q-\alpha+4}}\cdot \left(\frac 14 \frac{\partial}{\partial r} |\eta^{-1}\delta_r(\xi)|^4 \right)_{r=1} d\eta d\xi\\
    &+(\alpha-Q+1)\lambda \int_{\Omega}\int_{\Omega} \frac{f^{p-1}(\xi)f^{p-1}(\eta)} {|\eta^{-1}\xi|^{Q-\alpha+3}}\cdot\left(\frac 14 \frac{\partial}{\partial r} |\eta^{-1}\delta_r(\xi)|^4\right)_{r=1} d\eta d\xi\\
    :=&I+II.
\end{align*}
Estimating $I$ by \eqref{formula 2.1}, we have
\begin{align*}
    2\times I=&(\alpha-Q)\int_{\Omega}\int_{\Omega} \frac{f^{p-1}(\xi)f^{p-1}(\eta)} {|\eta^{-1}\xi|^{Q-\alpha+4}} \left\{|z-w|^2[x(x-u)+y(y-v)]\right.\\
    &\hspace{2.4cm}\left.+(t-s+2(yu-xv))(t+(yu-xv))\right\} d\eta d\xi\\
    &+(\alpha-Q)\int_{\Omega}\int_{\Omega} \frac{f^{p-1}(\xi)f^{p-1}(\eta)} {|\xi^{-1}\eta|^{Q-\alpha+4}} \left\{|w-z|^2[u(u-x)+v(v-y)]\right.\\
    &\hspace{2.4cm}\left.+(s-t+2(vx-uy))(s+(vx-uy))\right\} d\xi d\eta \\
    =&(\alpha-Q) \int_{\Omega}\int_{\Omega} \frac{f^{p-1}(\xi)f^{p-1}(\eta)} {|\eta^{-1}\xi|^{Q-\alpha}} d\eta d\xi.
\end{align*}
That is
\begin{equation*}
    I=\frac{\alpha-Q}2 \int_{\Omega}\int_{\Omega} \frac{f^{p-1}(\xi)f^{p-1}(\eta)} {|\eta^{-1}\xi|^{Q-\alpha}} d\eta d\xi.
\end{equation*}
Similarly,
\begin{equation*}
    II=\frac{\alpha-Q+1}2\lambda \int_{\Omega}\int_{\Omega} \frac{f^{p-1}(\xi)f^{p-1}(\eta)} {|\eta^{-1}\xi|^{Q-\alpha-1}} d\eta d\xi.
\end{equation*}
Combining the above into \eqref{formula 2.4}, we arrive at
\begin{align}\label{formula 2.2}
    &\int_\Omega f^{p-1}(\xi) Ef(\xi) d\xi\nonumber\\
    =&\frac{\alpha-Q}2 \int_{\Omega}\int_{\Omega} \frac{f^{p-1}(\xi)f^{p-1}(\eta)} {|\eta^{-1}\xi|^{Q-\alpha}} d\eta d\xi \nonumber\\
    &+\frac{\alpha-Q+1}2\lambda \int_{\Omega}\int_{\Omega} \frac{f^{p-1}(\xi)f^{p-1}(\eta)} {|\eta^{-1}\xi|^{Q-\alpha-1}} d\eta d\xi\nonumber\\
    =&\frac{\alpha-Q}2 \int_{\Omega} f^{p-1}(\xi) \left[f(\xi)-\lambda\int_\Omega \frac{f^{p-1}(\eta)} {|\eta^{-1}\xi|^{Q-\alpha-1}} d\eta\right] d\xi \nonumber\\
    &+\frac{\alpha-Q+1}2\lambda \int_{\Omega}\int_{\Omega} \frac{f^{p-1}(\xi)f^{p-1}(\eta)} {|\eta^{-1}\xi|^{Q-\alpha-1}} d\eta d\xi\nonumber\\
    =&\frac{\alpha-Q}2\int_\Omega f^p(\xi) d\xi+\frac\lambda 2\int_{\Omega}\int_{\Omega} \frac{f^{p-1}(\xi)f^{p-1}(\eta)} {|\eta^{-1}\xi|^{Q-\alpha-1}} d\eta d\xi.
\end{align}

On the other hand, by integration by part, we have
\begin{align}\label{formula 2.3}
    &\int_\Omega f^{p-1}(\xi) Ef(\xi)d\xi=\int_\Omega E(\frac {u^p(\xi)}p) d\xi=\frac 1p\int_{\partial\Omega}u^p(\xi) E\cdot\nu dS-\frac Qp\int_\Omega u^p(\xi) d\xi.
\end{align}
Hence, We deduce \eqref{Pohozaev identity} by combining \eqref{formula 2.2} and \eqref{formula 2.3}.
\end{proof}

\medskip

\noindent{\bf Proof of part (3) in Theorem \ref{main} (nonexistence part)}. If $f(\xi)$ is a non-negative $C^1(\overline{\Omega}) $ solution to \eqref{HB} for $\lambda \le 0$, then by Lemma \ref{poh}, we know that  $g(\xi)=f^{q-1}(\xi)$ satisfies
\begin{eqnarray}\label{SB-5}
-\frac{\lambda}2\int_{\Omega}\int_{\Omega} \frac{g^{p-1}(\xi)g^{p-1}(\eta)} {|\eta^{-1}\xi|^{Q-\alpha-1}}dydx +\frac 1p \int_{\partial \Omega}(E\cdot\nu) g^{p-1}(\xi)d\sigma\le 0.
\end{eqnarray}
Since $\Omega$ is $\delta$-starshaped domain about the origin, we have $E\cdot \nu>0$  on $\partial\Omega$.  If $\lambda<0$, then $g(\xi)\equiv 0 $ on $\Omega$. If $\lambda=0$, It follows immediately from \eqref{SB-5} that $g\equiv0$ on $\partial\Omega$.  Therefore, from \eqref{HB} we conclude that $g\equiv0$ on $\Omega$.  Part (3) in Theorem 1.1 is proved.\hfill $\Box$

\section{Existence result for subcritical case}\label{Sec supercritial exist}

To obtain the existence to equation \eqref{HB} with subcritical powers,  we need the following compactness lemma.

\begin{lemma}[Compactness]\label{compact lem}
For any compact domain $\Omega$, operator $I_\alpha: L^{\frac{2Q}{Q+\alpha}}(\Omega) \rightarrow L^r(\Omega)$, $r<\frac{2Q}{Q-\alpha}$, is compact. Namely, for any bounded sequence  $\{f_j\}_{j=1}^{+\infty} \subset L^{\frac{2Q}{Q+\alpha}}(\Omega)$, there exist a function $f\in L^{\frac{2Q}{Q+\alpha}}(\Omega)$ and a subsequence of $\{I_\alpha f_j(\xi)\}_{j=1}^{+\infty}$ which converges to $I_\alpha f$ in $L^r(\Omega)$.
\end{lemma}
\begin{proof}
Since the sequence $\{f_j\}$ is bounded in $L^{\frac{2Q}{Q+\alpha}}(\Omega)$, then there exist a subsequence (still denoted by $\{f_j\}$) and a function $f\in L^{\frac{2Q}{Q+\alpha}}(\Omega)$ such that
    $$f_j\rightharpoonup f\quad \text{weakly in} \quad L^{\frac{2Q}{Q+\alpha}}(\Omega).$$

Decompose $|\xi|^{\alpha-Q} =|\xi|^{\alpha-Q}\chi_{\{|\xi|>\rho\}} +|\xi|^{\alpha-Q}\chi_{\{|\xi|<\rho\}}$, where $\rho>0$ will be chosen later. Then,
\begin{align*}
    I_\alpha f_j(\xi)&=I_\alpha^1 f_j(\xi)+I_\alpha^2 f_j(\xi)\\
    &:=f_j*|\xi|^{\alpha-Q}\chi_{\{|\xi|>\rho\}} +f_j*|\xi|^{\alpha-Q}\chi_{\{|\xi|<\rho\}}.
\end{align*}

Noting $|\xi|^{\alpha-Q}\chi_{\{|\xi|>\rho\}}\in L^{\frac{2Q}{Q-\alpha}}$, then $I_\alpha^1 f_j(\xi)$ converges pointwise to $I_\alpha^1 f(\xi)$ by the weak convergence. On the other hand, since
    $$|I_\alpha^1 f_j(\xi)|\leq \|f_j\|_{L^{\frac{2Q}{Q+\alpha}}} \bigl\||\xi|^{\alpha-Q}\chi_{\{|\xi|>\rho\}}\bigr\|_{L^{\frac{2Q}{Q-\alpha}}}\leq C(\rho),$$
where $C(\rho)$ is independent of $f_j$, then the dominated convergence deduce that
\begin{equation}\label{formula 3.2}
    I_\alpha^1 f_j\rightarrow I_\alpha^1 f \quad \text{strongly in} \quad L^r(\Omega).
\end{equation}

Next, we analyze the convergence of $\{I_\alpha^2 f_j\}$ by the Young inequality. Take $s=\left(\frac 1r+\frac{Q-\alpha}{2Q}\right)^{-1}$. Then, $s<\frac Q{Q-\alpha}$ and $\bigl\||\xi|^{\alpha-Q}\chi_{\{|\xi|<\rho\}}\bigr\|_{L^s}<C\rho^\beta$ with $\beta=Q(\frac 1s-\frac{Q-\alpha}Q)$. By the Young inequality, we have
\begin{equation}\label{formula 3.3}
    \|I_\alpha^2(f_j-f)\|_{L^r}\le C\|f_j-f\|_{L^{\frac{2Q}{Q+\alpha}}} \bigl\||\xi|^{\alpha-Q}\chi_{\{|\xi|<\rho\}}\bigr\|_{L^s}\le C\rho^\beta.
\end{equation}
By now, through choosing first $\rho$ small and then $j$ large, we deduce by \eqref{formula 3.2} and \eqref{formula 3.3} that
  \[
  I_\alpha f_j\rightarrow I_\alpha f \quad\text{strongly in}\quad L^r(\Omega).
  \]
  The Lemma is proved.
\end{proof}

%

Based on the Lemma \ref{compact lem}, we can obtain the existence result (part (1) in Theorem \ref{main}). For simplicity, we only present the proof for $\lambda=0$.
\begin{lemma}\label{Bound- attained}
For $q>{q_\alpha}$, supremum
\begin{equation*}
D_{\alpha,q}(\Omega):=\sup_{f\in L^q(\Omega)\setminus\{0\}} \frac{\int_{\Omega} \int_{\Omega}f(\xi)|\eta^{-1}\xi|^{-(Q-\alpha)} f(\eta) d\eta d\xi}{\|f\|^2_{L^q(\Omega)}}
\end{equation*}
 is attained by some nonnegative  function in $L^q(\Omega).$
\end{lemma}
\begin{proof}
First, by HLS inequality \eqref{class-HLS-1} or \eqref{class-HLS}, we know that
    $$D_{\alpha,q}(\Omega)
    \le \sup_{f\in L^q(\Omega)\setminus\{0\}}\frac{\|I_\alpha f\|_{L^{q'}(\Omega)}}{\|f\|_{L^q(\Omega)}} < +\infty.$$
Choosing a nonnegative maximizing sequence $\{f_j\}_{j=1}^\infty\subset L^q(\Omega)$ satisfying $\|f_j\|_L^q(\Omega)=1$ and
    $$\lim_{j\rightarrow+\infty} \int_{\Omega} \int_{\Omega} f_j(\xi)|\eta^{-1}\xi|^{-(Q-\alpha)} f_j(\eta) d\eta d\xi=D_{\alpha,q}(\Omega).$$
Combining the boundedness of $\{f_j\}$ in $L^q(\Omega)$ and the compactness of the operator $I_{\alpha,\Omega}$, we deduce by Lemma \ref{compact lem} that there exists a subsequence (still denoted as $\{f_j\}$) and $f_*\in L^q(\Omega)$ such that
\begin{gather*}
    f_j\rightharpoonup f_*\quad\text{weakly\, in}\quad L^q(\Omega),\\
    I_{\alpha, \Omega} f_j\rightarrow I_{\alpha,\Omega} f_*\quad\text{strongly\, in}\quad L^{q'}(\Omega).
\end{gather*}
Thus, $\|f_*\|_{L^q(\Omega)}\le\liminf_{j\to\infty}\|f_j\|_{L^q(\Omega)}$ and
    $$\lim_{j\rightarrow+\infty}\langle I_{\alpha,\Omega}f_j,f_j\rangle=\langle I_{\alpha,\Omega} f,f\rangle.$$
Then,
    $$D_{\alpha,q}(\Omega)=\lim_{j \to \infty}\frac{\langle I_{\alpha, \Omega} f_j,f_j\rangle}{\|f_j\|^2_{L^q(\Omega)}}\le\frac{\langle I_{\alpha, \Omega} f_*,f_*\rangle}{\|f_*\|^2_{L^q(\Omega)}},$$
namely, $f_*$ is a maximizer.
\end{proof}
\smallskip

It is easy to see that  the  maximizer for energy $D_{\alpha, q}(\Omega)$, up to a constant multiplier,  satisfies following equation:
\begin{equation}\label{EL-equ-2}
f^{q-1}(\xi)=\int_\Omega\frac{f(\eta)}{|\eta^{-1}\xi|^{Q-\alpha}}d\eta, \quad \xi\in \overline{\Omega}.
\end{equation}
Let $g(\xi)=f^{q-1}(\xi)$, and then \eqref{EL-equ-2} is changed into the form
\begin{equation}\label{EL-equ-22}
g(\xi)=\int_\Omega\frac{g^{q'-1}(\eta)}{|\eta^{-1}\xi|^{Q-\alpha}}d\eta, \quad \xi\in \overline{\Omega}
\end{equation}
for $q'<\frac{2Q}{Q-\alpha}=p_\alpha.$ To complete the proof of Part (1) in Theorem \ref{main} we need to show that $g \in \Gamma^\alpha(\overline \Omega)$.

\begin{lemma}[Regularity]\label{regularity-lemma}
Suppose that $g\in L^{q'}(\Omega)$ is a  positive solution to \eqref{EL-equ-22}. If $ q'<{p_\alpha},$ then $g \in \Gamma^\alpha(\overline \Omega)\subset C^{\alpha/2}(\overline\Omega).$
\end{lemma}
\begin{proof}
{\bf Step 1.}   we show $g\in L^\infty(\overline\Omega)\cup C(\overline\Omega)$.

For proving $g \in L^\infty(\overline \Omega)$, it is necessary to prove that there exists some constant $s^*>0$ such that $g\in L^{s^*}(\Omega)$ and $\frac {s^*}{q'-1}>\frac Q\alpha$. In fact, if there exists such $s^*$, then
    $$g(\xi)\leq \|g\|_{L^{s^*}(\Omega)}^{q'-1}\left(\int_\Omega |\eta^{-1}\xi|^{(\alpha-Q)(s^*/(q'-1))'} d\eta\right)^{\frac 1{(s^*/(q'-1))'}}\leq C \|g\|_{L^{s^*}(\Omega)}^{q'-1},$$
where $(s^*/(q'-1))'$ is the conjugate number of $(s^*/(q'-1))$. Hence, $g\in L^\infty(\Omega)$, which leads to $g\in C(\overline\Omega)$ by the dominant convergence theorem.

i) If $q'<\frac Q{Q-\alpha}$, we can take $s^*=q'$.

ii) If $q'=\frac Q{Q-\alpha}$, then $g\in L^{q_1}(\Omega)$ with $q_1=(1-\frac 1k)\frac Q{Q-\alpha}$ and $k=\left[\frac Q{Q-\alpha}\right]+1$. By \eqref{EL-equ-22} and HLS inequality \eqref{class-HLS-1}, we know that $g\in L^{s^*}(\Omega)$ with $\frac 1{s^*}=\frac {q'-1}{q_1}-\frac\alpha Q=\frac \alpha{(k-1)Q}$ and $\frac {s^*}{q'-1}>\frac Q\alpha$.

iii) For the case $\frac{Q}{Q-\alpha}<q'<\frac{2Q}{Q-\alpha}$, we will find the constant $s^*$ by the following iteration process. By \eqref{EL-equ-22} and HLS inequality \eqref{class-HLS-1}, we have
\begin{equation}\label{it18}
||g||_{L^s(\Omega)}=||I_{\alpha, \Omega} (g^{q'-1})||_{L^s(\Omega)} \le C ||g^{q'-1}||_{L^t(\Omega)}
\end{equation}
for $1/s=1/t-\alpha/Q$. We use the above inequality to do iteration. First, choose $t=q'/(q'-1):=t_1/(q'-1)$, and let $t_2=s$. Since $q'<2Q/(Q-\alpha)$, it is easy to check that $t_2>2Q/(Q-\alpha)$.

Iterate the above: let $t=t_i/(q'-1)$, then $t_{i+1}=s$ for $i=1,2,\cdots$. Note that $q'-2<2\alpha/(Q-\alpha),$ and $1/t_i<(Q-\alpha)/(2Q)$. Thus, one can check  that $t_{i+1}>t_i$ when  $t_{i+1}$ is positive.  So, after iterating certain times, say, $k_0$ times, we shall have: $ \frac{q'-1}{t_{k_0}}>\frac\alpha Q$ and $\frac{q'-1}{t_{k_0+1}}\le\frac\alpha Q$.

If $\frac{q'-1}{t_{k_0+1}}<\frac\alpha Q$, we can take $s^*=t_{k_0+1}$.

If $\frac{q'-1}{t_{k_0+1}}=\frac\alpha Q$, then $u\in L^{t_{k_0+1}}(\Omega)=L^{(q'-1)\frac Q\alpha}(\Omega)$. By the boundedness of $\Omega$ and the H\"{o}lder inequality, we know that $u\in L^{q_2}(\Omega)$ with $q_2=(1-\frac 1k)(q'-1)\frac Q\alpha$ and $k=\left[\frac{2Q}{Q-\alpha}\right]+1$. Using \eqref{EL-equ-22} and HLS inequality \eqref{class-HLS-1}, we get $u\in L^{s^*}(\Omega)$ with $\frac 1{s^*}=\frac{q'-1}{q_2}-\frac\alpha Q$ and $\frac{s^*}{q'-1}>\frac Q\alpha$.
\medskip

\noindent{\bf Step 2.}  We show $g\in\Gamma^\alpha(\overline\Omega)$. we divide it into four cases.

\textbf{Case i)} $0<\alpha<1$.

Since $g\in L^\infty(\Omega)$,  we have
\begin{equation}\begin{split}\label{formula 4.12-0}
    |g(\xi\gamma)-g(\xi)|=&\left| \int_\Omega (g(\xi\eta^{-1}))^{q'-1} (|\eta\gamma|^{\alpha-Q} -|\eta|^{\alpha-Q}) d\eta\right|\\
    \leq &\|g\|_{L^\infty(\Omega)}^{q'-1} \int_\Omega \left| |\eta\gamma|^{\alpha-Q} -|\eta|^{\alpha-Q}\right| d\eta.
\end{split}\end{equation}
By the Proposition \ref{pro 1.15 of Folland}, we can estimate
\begin{align}\label{formula 4.12}
    &\int_{\Omega\cap\{|\eta|\geq 2|\gamma|\}} \left| |\eta\gamma|^{\alpha-Q} -|\eta|^{\alpha-Q}\right| d\eta\nonumber\\ \leq & C\int_{\Omega\cap\{|\eta|\geq 2|\gamma|\}}|\gamma||\eta|^{\alpha-Q-1} d\eta\leq C|\gamma|^\alpha.
\end{align}
On the other hand, there exists some $B\geq 2$ such that $|\eta\gamma|\leq B|\gamma|$ since $|\eta|\leq 2|\gamma|$. By Proposition \ref{Pro triangle ineq},
\begin{align}\label{formula 4.13}
    &\int_{\Omega\cap\{|\eta|\leq 2|\gamma|\}} \left| |\eta\gamma|^{\alpha-Q} -|\eta|^{\alpha-Q}\right| d\eta \nonumber\\
    \leq &\int_{\Omega\cap\{|\eta\gamma|\leq 3C|\gamma|\}} |\eta\gamma|^{\alpha-Q} d\eta +\int_{\Omega\cap\{|\eta|\leq 2|\gamma|\}} |\eta|^{\alpha-Q} d\eta\leq C|\gamma|^\alpha.
\end{align}
Substituting \eqref{formula 4.12} and \eqref{formula 4.13} into \eqref{formula 4.12-0}, we have
    $$|g(\xi\gamma)-g(\xi)|\leq C\|g\|_{L^\infty(\Omega)}^{q'-1}|\gamma|^\alpha,$$
i.e., $g\in \Gamma^\alpha(\Omega)$.

\textbf{Case ii)} $\alpha=1$.

By an argument similar to Case i), we can get
    $$|g(\xi\gamma)+g(\xi\gamma^{-1})-2g(\xi)|\leq C\|g\|_{L^\infty(\Omega)}^{q'-1}|\gamma|^\alpha,$$
i.e., $g\in \Gamma^1(\Omega)$.

\textbf{Case iii)} $\alpha=1+\alpha'$ with $0<\alpha'\leq 1$.

Since $X_j g=g^{q'-1}*X_j(|\xi|^{\alpha-Q})\ (1\leq j\leq n)$ and $X_j(|\xi|^{\alpha-Q})$ is homogeneous of degree $\alpha-Q-1$ by Proposition \ref{pro 8.1 of FS}, we can repeat the argument of Case $0<\alpha\leq 1$ and show that $X_jg\in \Gamma^{\alpha'}(\Omega)$. Similarly, one can show that $Y_j g\in \Gamma^{\alpha'}(\Omega),\ 1\leq j\leq n$.

Following, we will prove that $g\in\Gamma^{\alpha'}(\Omega)$. Since $g\in L^\infty(\Omega)$ and $\Omega$ is bounded, we have $g^{q'-1}\in L^Q(\Omega)$. If $0<\alpha'<1$,
\begin{align}\label{formula 4.14}
    &|g(\xi\gamma)-g(\xi)|\nonumber\\
    \leq &\|g^{q'-1}\|_{L^Q(\Omega)} \left(\int_{\Omega\cap\{|\eta|\geq 2|\gamma|\}}\left| |\eta\gamma|^{\alpha-Q}-|\eta|^{\alpha-Q}\right|^{Q'} d\eta\right)^{1/Q'}\nonumber\\
    &+\|g^{q'-1}\|_{L^Q(\Omega)} \left(\int_{\Omega\cap\{|\eta|\leq 2|\gamma|\}}\left| |\eta\gamma|^{\alpha-Q}-|\eta|^{\alpha-Q}\right|^{Q'} d\eta\right)^{1/Q'},
\end{align}
where $Q'=\frac Q{Q-1}$. Similar to Case $0<\alpha<1$, we have
\begin{align}\label{formula 4.15}
    &\left(\int_{\Omega\cap\{|\eta|\geq 2|\gamma|\}}\left| |\eta\gamma|^{\alpha-Q}-|\eta|^{\alpha-Q}\right|^{Q'} d\eta\right)^{1/Q'}\nonumber\\
    \leq & C \left(\int_{\{|\eta|\geq 2|\gamma|\}} \left(|\gamma||\eta|^{\alpha-Q-1}\right)^{Q'} d\eta\right)^{1/Q'}\nonumber\\
    \leq & C|\gamma|\cdot|\gamma|^{\alpha-Q-1+\frac Q{Q'}}=C|\gamma|^{\alpha'},
\end{align}
and
\begin{align}\label{formula 4.16}
    &\left(\int_{\Omega\cap\{|\eta|\leq 2|\gamma|\}}\left| |\eta\gamma|^{\alpha-Q}-|\eta|^{\alpha-Q}\right|^{Q'} d\eta\right)^{1/Q'}\nonumber\\
    \leq & C\left(\int_{\{|\eta|\leq B|\gamma|\}}|\eta|^{(\alpha-Q)Q'} d\eta\right)^{1/Q'}\nonumber\\
    \leq & C|\gamma|^{\alpha-Q+\frac Q{Q'}}=C|\gamma|^{\alpha'}.
\end{align}
Combining \eqref{formula 4.14}, \eqref{formula 4.15} and \eqref{formula 4.16} leads to $g\in\Gamma^{\alpha'}(\Omega)$.

If $\alpha'=1$, the proof can be completed similarly.

\textbf{Case iv)} $\alpha=k+\alpha'$ with $0<\alpha'\leq 1$ and $k=2,3,\cdots$.

This case can be discussed with a similar argument with Case iii).
\end{proof}

\section{Existence result for critical case}\label{Sec critical exist}

Now we shall establish  the existence results for \eqref{HB} with $\lambda>0$ and $q=q_\alpha$. To this end, we consider
\begin{equation*}
Q_\lambda(\Omega):=\sup_{f\in L^{q_\alpha}(\Omega)\setminus\{0\}}\frac{\int_{\Omega} \int_{\Omega}f(\xi)(|\eta^{-1}\xi|^{-(Q-\alpha)}+\lambda |\eta^{-1}\xi|^{-(Q-\alpha-1)}) f(\eta) d\eta d\xi}{\|f\|^2_{L^{q_\alpha}(\Omega)}}.
\end{equation*}
 Notice that the corresponding Euler-Lagrange equation for extremal functions, up to a constant multiplier,  is  integral equation \eqref{HB} for $q=q_\alpha$.

First, we show
\begin{lemma}\label{Lm-inequ}
$Q_\lambda(\Omega)>D_{n,\alpha}$ for all $\lambda>0$.
\end{lemma}
\begin{proof} Let $\zeta\in \Omega$.
For small positive $\epsilon$ and a fixed $R>0$ so that $\Sigma_R(\zeta) \subset \Omega$, we define
\begin{equation*}
    \tilde{f}_\epsilon(\xi)= \begin{cases}
    f_\epsilon(\xi) &\quad \xi\in \Sigma_{R}(\zeta)\subset\Omega,\\
    0&\quad \xi\in\mathbb{H}^n\backslash \Sigma_{R}(\zeta), \end{cases}
\end{equation*}
where $f_\epsilon$ is given by \eqref{fe}.
Obviously, $\tilde{f}_\epsilon\in L^{q_\alpha}(\mathbb{H}^n).$ Thus,  similar to the proof of Proposition \ref{prop2-1},  we have
\begin{align*}\label{BE-1}
    &\int_{\Omega} \int_{\Omega} \big(\frac{1}{|\eta^{-1}\xi|^{Q-\alpha}} +\frac{\lambda}{|\eta^{-1}\xi|^{Q-\alpha-1}}\big) \tilde{f}_\epsilon(\xi)\tilde{f}_\epsilon(\eta) d\xi d\eta\\
    =&\int_{\mathbb{H}^n} \int_{\mathbb{H}^n} \frac{1}{|\eta^{-1}\xi|^{Q-\alpha}}{f_\epsilon}(\xi){f_\epsilon}(\eta) d\xi d\eta\\
    &-2\int_{\mathbb{H}^n} \int_{\Sigma_R^C(\zeta)} \frac{ f_\epsilon(\xi) f_\epsilon(\eta)}{|\eta^{-1}\xi|^{Q-\alpha}} d\xi d\eta +\int_{\Sigma_R^C(\zeta)} \int_{\Sigma_R^C(\zeta)} \frac{ f_\epsilon(\xi) f_\epsilon(\eta)}{|\eta^{-1}\xi|^{Q-\alpha}} d\xi d\eta\\
    &+\lambda \int_{\Sigma_R(\zeta)} \int_{\Sigma_R(\zeta)}\frac{f_\epsilon(\xi)f_\epsilon(\eta)} {|\eta^{-1}\xi|^{Q-\alpha-1}}\big) d\xi d\eta\\
    =&D_{n,\alpha}\|f_\epsilon\|^2_{L^{q_\alpha}(\mathbb{H}^n)}-I_1+I_2+I_3,
\end{align*}
where
\begin{align*}
    I_1:=&2\int_{\mathbb{H}^n} \int_{\Sigma_R^C(\zeta)} \frac{ f_\epsilon(\xi) f_\epsilon(\eta)}{|\eta^{-1}\xi|^{Q-\alpha}} d\xi d\eta,\\
    I_2:=&\int_{\Sigma_R^C(\zeta)} \int_{\Sigma_R^C(\zeta)} \frac{ f_\epsilon(\xi) f_\epsilon(\eta)}{|\eta^{-1}\xi|^{Q-\alpha}} d\xi d\eta,\\
    I_3:=&\lambda \int_{\Sigma_R(\zeta)} \int_{\Sigma_R(\zeta)}\frac{f_\epsilon(\xi)f_\epsilon(\eta)} {|\eta^{-1}\xi|^{Q-\alpha-1}}\big) d\xi d\eta.
\end{align*}
For $I_3$, we have
\begin{align*}
    I_3:=&\lambda \int_{\Sigma_R(0)} \int_{\Sigma_R(0)} \frac{\epsilon^{-\frac{Q+\alpha}2}H(\delta_{\epsilon^{-1}}(\xi)) \epsilon^{-\frac{Q+\alpha}2}H(\delta_{\epsilon^{-1}}(\eta))} {|\eta^{-1}\xi|^{Q-\alpha-1}} d\xi d\eta\\
    =&\lambda\epsilon\int_{\Sigma_{R/\epsilon}(0)} \int_{\Sigma_{R/\epsilon}(0)} \frac{H(\xi)H(\eta)}{|\eta^{-1}\xi|^{Q-\alpha-1}} d\xi d\eta\ge C_0\lambda\epsilon.
\end{align*}
So, for $\lambda>0,$ and small enough $\epsilon$, we have
\begin{eqnarray*}
-I_1+I_2+I_3\ge -C_1(\frac{R}\epsilon)^{-Q}+C_0\lambda\epsilon =\epsilon\big(-\frac{C_1}{R}(\frac\epsilon R)^{Q-1}+C_0\lambda\big)>0.
\end{eqnarray*}
\end{proof}

To show the existence of weak solution, we first establish the following criterion for the existence of maximizer for energy $Q_\lambda(\Omega)$.

\begin{proposition}\label{sub-existence-2}
If $Q_\lambda(\Omega)>D_{n,\alpha}$ for a given $\lambda>0$,  $Q_\lambda(\Omega) $ is achieved by a positive function $f_* \in L^{q_\alpha}(\Omega).$
\end{proposition}

To complete the proof of {\bf Proposition \ref{sub-existence-2}}, we will adapt the method of blow-up analysis. Namely, we will prove firstly the existence of \eqref{HB} with subcritical exponent problem and then get the existence of \eqref{HB} with critical exponent by compactness.

\medskip

Consider
$$Q_{\lambda, q}(\Omega)=\sup_{f\in L^{q}(\Omega)\setminus\{0\}}\frac{ \int_{\Omega} \int_{\Omega} f(\xi)(|\eta^{-1}\xi|^{-(Q-\alpha)}+\lambda |\eta^{-1}\xi|^{-(Q-\alpha-1)}) f(\eta) d\xi d\eta}{\|f\|^2_{L^{q}(\Omega)}}
$$
for $q>q_\alpha.$ Similar to the proof of Lemma \ref{Bound- attained}, we easily show that the supreme is attained by a positive function $f_q$, which satisfies
the subcritical equation
\begin{equation}\label{sub-equ2}
Q_{\lambda, q}(\Omega) f^{q-1}(\xi)=\int_\Omega\frac{f(\eta)}{|\eta^{-1}\xi|^{Q-\alpha}}d\eta+ \lambda \int_\Omega\frac {f(\eta)}{|\eta^{-1}\xi|^{Q-\alpha-1}}d\eta,\quad \xi\in \overline \Omega
\end{equation}
with the constraint $||f_q||_{q}=1.$ Further, we can show easily that $f_q \in \Gamma^\alpha(\overline \Omega)$  and  $Q_{\lambda, q}(\Omega) \to Q_\lambda(\Omega)$ for $q \to q_\alpha^+$.

\begin{lemma}\label{lem3-3}
Let $f_q>0$ being   maximum energy solutions to \eqref{sub-equ2} for $q\in (q_\alpha, 2).$ If  there exists some $q_0\in (q_\alpha,2)$ such that $Q_{\lambda, q}(\Omega)\ge D_{n,\alpha}+\epsilon$ for any $q\in (q_\alpha,q_0)$, then the sequence $\{f_q\}_{q_\alpha<q<q_0}$ is uniformly bounded in $\Omega$.
\end{lemma}
\begin{proof}
We only need to show $\lim_{q \to q_\alpha^+} ||f_p||_{C^0({\overline \Omega})} \le C.$ We prove this by contradiction. Suppose not. Let $f_q(\xi_q)=\max_{\overline \Omega} f_q(\xi)$. Then $f_q(\xi_q) \to \infty$ for $q \to q_\alpha^+$. Let
    $$\mu_q= f_q^{-\frac {2-q}\alpha}(\xi_q), \, \, \mbox{ and} \, \, \Omega_q=\delta_{\mu_q^{-1}}(\xi_q^{-1}\Omega):=\{\varsigma \ | \ \varsigma=\delta_{\mu_q^{-1}}(\xi_q^{-1}\xi) \ \mbox{for} \, \xi \in \Omega\}.$$
Define
\begin{equation}\label{blowup-1}
    g_q(\varsigma)=\mu_q^{\frac{\alpha}{2-q}} f_q(\xi_q\delta_{\mu_q}(\varsigma)), \,  \ \, \, \mbox{for} \, \, \varsigma \in \Omega_q.
\end{equation}
Then, $g_q$ satisfies
\begin{equation}\label{sub-equ3}
    Q_{\lambda, q}(\Omega) g_q^{q-1}(\varsigma)=\int_{\Omega_q}\frac{g_q(\eta)} {|\eta^{-1}\varsigma|^{Q-\alpha}}d\eta+ \lambda \int_\Omega\frac {f_q(\xi_q)^{1-q} \cdot f_q(\eta)}{|\eta^{-1}\xi|^{Q-\alpha-1}}d\eta
\end{equation}
and $g_q(0)=1$, $g_q(\varsigma) \in (0, 1].$

Choose $q_\delta=q-\delta$ for some small $\delta>0$. We can check that
\begin{align}\label{small-term}
    \int_\Omega\frac {f_q(\xi_q)^{1-q} \cdot f_q(\eta)}{|\eta^{-1}\xi|^{Q-\alpha-1}}d\eta
    =& f_q(\xi_q)^{-\delta} \int_\Omega\frac {f_q(\xi_q)^{1-q_\delta} \cdot f_q(\eta)}{|\eta^{-1}\xi|^{Q-\alpha-1}}d\eta \nonumber\\
    \le &  f_q(\xi_q)^{-\delta}  \int_\Omega\frac {f_q(\eta)^{2-q_\delta}} {|\eta^{-1}\xi|^{Q-\alpha-1}}d\eta \nonumber\\
    \le & C f_q(\xi_q)^{-\delta} \cdot ||f_q||_{q}^{2-q_\delta} \to 0, \ \  \ \mbox{as} \ \ \  q\to q_\alpha^+.
\end{align}

For $\varsigma\in \mathbb{H}^n$ and $R>B|\varsigma|$, where the constant $B$ is larger than the constant in {\bf Proposition \ref{Pro triangle ineq}}, we have
\begin{align}\label{tail-term}
    &\int_{\Omega_q\setminus B_R}\frac{g_q(\eta)} {|\eta^{-1}\varsigma|^{Q-\alpha}}d\eta\leq C\int_{\Omega_q\setminus B_R}\frac{g_q(\eta)} {|\eta|^{Q-\alpha}}d\eta\nonumber\\
    =& C\mu_q^{\frac\alpha{2-q}-\alpha} \int_{\Omega\setminus B_{R\mu_q}(\xi_q)} \frac{f_q(\eta)} {|\xi_q^{-1}\eta|^{Q-\alpha}} d\eta\nonumber\\
    \leq & C\mu_q^{\frac\alpha{2-q}-\alpha} \|f_q\|_q \left(\int_{\Omega\setminus B_{R\mu_q}(\xi_q)} |\xi_1^{-1}\eta|^{(\alpha-Q)\frac q{q-1}} d\eta\right)^{\frac{q-1}q}\nonumber\\
    \leq & C \mu_q^{\frac\alpha{2-q}-\alpha} (R\mu_q)^{\alpha-\frac Qq}\rightarrow 0,
\end{align}
as $q\rightarrow q_\alpha^+$ and $R\rightarrow+\infty$.

As $q\to q_\alpha^+$, there are two cases:
\noindent {\textbf{Case 1.}} \ $\Omega_q \to \mathbb{H}^n,$ and $\bm{g_q(z) \to g(z)}$ pointwise in $\bm{\mathbb{H}^n}$, where $g(z)$ satisfies (from \eqref{sub-equ3}, and estimates \eqref{small-term} and \eqref{tail-term}):
\begin{equation}\label{blowup-3}
    Q_\lambda(\Omega) g^{q_\alpha-1}(\varsigma)= \int_{\mathbb{H}^n} \frac {g(\eta)}{|\eta^{-1}\varsigma|^{Q-\alpha}} d\eta, \ \ \ \ g(0)=1.
\end{equation}
Also, a direct computation yields
    $$1=\int_\Omega f_q^q(\xi) d\xi= \mu_q^{Q-\frac q{2-q}\alpha} \int_{\Omega_q} g_q^q(\varsigma) d\varsigma \ge \int_{\Omega_q} g_q^q d\varsigma.$$
Thus $\int_{\mathbb{H}^n} g^{q_\alpha} d\varsigma \le 1.$  Combining this with \eqref{blowup-3}, we have
    $$D_{n,\alpha}+\epsilon\le Q_\lambda(\Omega)= \frac{\int_{\mathbb{H}^n} \int_{\mathbb{H}^n} \frac {g(\xi) g(\eta)}{|\eta^{-1}\xi|^{Q-\alpha}} d\xi d\eta}{\|g\|^{q_\alpha}_{L^{q_\alpha}(\mathbb{H}^n)}} \le \frac{\int_{\mathbb{H}^n} \int_{\mathbb{H}^n} \frac {g(\xi) g(\eta)}{|\eta^{-1}\xi|^{Q-\alpha}} d\xi d\eta}{\|g\|^{2}_{L^{q_\alpha}(\mathbb{H}^n)}}\le D_{n,\alpha}.$$
Contradiction!

\textbf{Case 2.} \  $\Omega_q \to \Omega_{q_\alpha}$, where $\Omega_{q_\alpha}$ is some subset of $\mathbb{H}^n$ satisfying $\omega_{q_\alpha}\neq\mathbb{H}^n$,
 $g_q(z) \to g(z) $ pointwise in $\Omega_{q_\alpha}$, where $g(z)$ satisfies
(from \eqref{sub-equ3}, and estimates \eqref{small-term} and \eqref{tail-term}):
\begin{equation}\label{blowup-2}
    Q_\lambda(\omega) g^{q_\alpha-1}(\varsigma)= \int_{\Omega_{q_\alpha}} \frac {g(\eta)}{|\eta^{-1}\varsigma|^{Q-\alpha}} d\eta, \ \ \ \ g(0)=1.
\end{equation}
Similarly, we know
$\int_{\Omega_{q_\alpha}} g^{q_\alpha} d\eta \le 1.$  Combining this with \eqref{blowup-2}, we have
$$
    D_{n,\alpha}+\epsilon\le Q_\lambda= \frac{\int_{\Omega_{q_\alpha}} \int_{\Omega_{q_\alpha}} \frac {g(\xi) g(\eta)}{|\eta^{-1}\xi|^{Q-\alpha}} d\xi d\eta} {\|g\|^{q_\alpha}_{L^{q_\alpha}(\Omega_{q_\alpha})}} \le \frac{\int_{\mathbb{H}^n} \int_{\mathbb{H}^n} \frac {\tilde g(\xi) \tilde g(\eta)}{|\eta^{-1}\xi|^{Q-\alpha}} d\xi d\eta}{\|\tilde g\|^{2}_{L^{q_\alpha}(\mathbb{H}^n)}}\le D_{n,\alpha}.
$$
Contradiction!
\end{proof}

\medskip

\noindent{\bf Proof of Proposition \ref{sub-existence-2}}. \   Let $f_q>0$ be solutions to \eqref{sub-equ2} for $q\in (q_\alpha, 2)$, which are also the maximal functions to energy $Q_{\lambda, q}(\Omega)$. Then $||f_q||_{L^\infty(\overline \Omega)} \le C$ by {\bf Lemma \ref{lem3-3}}, which yields $f_q$ is uniformly bounded and equi-continuous  due to equation  \eqref{sub-equ2}. Thus $f_q \to f_*$ as $q \to q_\alpha$ in $C^{0}(\overline \Omega),$ and $f_*$ is the energy maximizer for $Q_\lambda(\Omega)$.

\medskip

\noindent{\bf Completion of the Proof of Theorem \ref{main} (existence for critical case)}. From Lemma \ref{Lm-inequ} and Proposition \ref{sub-existence-2}, we know that there exists a positive solution $f_*\in C^0(\overline\Omega)$ of \eqref{HB} with critical exponent. Moreover, by a similar argument of the second part of Lemma \ref{regularity-lemma}, we have $f_*\in\Gamma^\alpha(\Omega)$.

\medskip
\noindent {\bf Acknowledgements}\\
\noindent The authors would like to thank Professor Meijun Zhu and Professor Jingbo Dou for their valuable comments.
The project is supported by  the
National Natural Science Foundation of China (Grant No. 11201443) and Natural Science Foundation of Zhejiang Province(Grant No. LY18A010013).

\end{document}